\documentclass[12pt]{amsart}
\usepackage[T2A]{fontenc}
\usepackage[utf8]{inputenc}   
\usepackage[russian, english]{babel}
\usepackage{enumerate}

\usepackage{hyperref}
\usepackage{amsmath}
\usepackage{amsthm}
\usepackage{amsfonts}
\usepackage{amssymb}
\usepackage{cite}

\usepackage{mathptmx} 
\usepackage[scaled=0.90]{helvet} 
\usepackage{courier} 
\usepackage[T1]{fontenc}
\usepackage{color}

\usepackage{bbm}
\usepackage{graphicx}
\DeclareGraphicsExtensions{.eps}

\newtheorem{theorem}{Theorem}[section]
\newtheorem{utv*}{Proposition}
\newtheorem{hyp*}{Conjecture}
\newtheorem*{example*}{Example}
\newtheorem{lemma}[theorem]{Lemma}

\newtheorem*{th*}{Theorem}

\theoremstyle{definition}

\newtheorem{defin}{Definition}
\newcommand{\av}[2]{\langle #1\rangle_{_{\scriptstyle #2}}}

\def\sli{\sum\limits}
\def\ili{\int\limits}

\def\la{\lambda}

\def\R{\mathbb{R}}

\def\ep{\varepsilon}
\def\vf{\varphi}

\def\cyr{\fontencoding{OT2}\fontfamily{wncyr}\selectfont}
\DeclareTextFontCommand{\textcyr}{\cyr}

{\end{list}}

\newcounter{vremennyj}

\title[Strong Muckenhoupt and Reverse H\"older weights]{Dimension free properties of strong Muckenhoupt and Reverse H\"older weights for Radon measures}
\date{\today}
\begin{document}
\author{O. Beznosova}
\address{Department of Mathematics, University of Alabama}
\email{ovbeznosova@ua.edu}
\author{A. Reznikov}
\address{Department of Mathematics, Florida State University}
\email{reznikov@math.fsu.edu}
\begin{abstract}
In this paper we prove self-improvement properties of strong Muckenhoupt and Reverse H\"older weights with respect to a general Radon measure on $\R^n$. We derive our result via a Bellman function argument. An important feature of our proof is that it uses only the Bellman function for the one-dimensional problem for Lebesgue measure; with this function in hand, we derive dimension free results for general measures and dimensions.
\end{abstract}

\maketitle

\vspace{4mm}


\vspace{2mm}


\vspace{2mm}

\normalsize
\section{Introduction}
It is well known that Muckenhoupt weights on a real line with respect to the Lebesgue measure satisfy self-improvement properties in the following sense: for $p>q$ we always have $A_q \subset A_p$; but also for any function $w\in A_p$, there is an $\ep>0$ such that $w\in A_{p-\ep}$ (we refer to Definition \ref{defin1} for precise definitions). Besides that, there always exists a $q$ such that $w\in RH_q$. These self-improvement properties allow one to prove many important results in harmonic analysis, see, e.g., \cite{garcia} or a more recent paper \cite{perez2}. In \cite{perez1}, the authors considered {\it strong Muckenhoupt classes}; in particular, it was proven that for a Radon measure $\mu$ on $\R^n$ which is absolutely continuous with respect to the Lebesgue measure $dx$, any weight $w\in A_p^*$ satisfies a Reverse H\"older property with an exponent that does not depend on the dimension $n$. 

For $p>1$, we say that $w$ belongs to the {\it strong Muckenhoupt class with respect to $\mu$}, $w\in A_p^*$, if there exists a number $Q>1$ such that for any rectangular box $R\subset \R^n$ with edges parallel to axis, we have
$$
\av{w}{R}\av{w^{-1/(p-1)}}{R}^{p-1} \leqslant Q,
$$
where $\av{\vf}{R}$ denotes the average of the function $\vf$ over $R$:
$$
\av{\vf}{R}:=\frac{1}{\mu(R)}\ili_R \vf(x)d\mu(x).
$$
For $p>1$, we say that $w$ belongs to the {\it strong Reverse H\"older class with respect to $\mu$}, $w\in RH_p^*$, if there exists a constant $Q>1$ such that for any rectangular box $R$ with edges parallel to axis, we have
$$
\av{w^p}{R}^{1/p}\leqslant Q\av{w}{R}.
$$
We proceed with the following definition.
\begin{defin}\label{defin1}
Let $\mu$ be a Radon measure on $\R^n$ and $w$ be a function which is positive $\mu$-a.e. For $p>1$, we denote the {\it $A_p^*$-characteristic of $w$} by
$$
[w]_p:=\sup_R \av{w}{R}\av{w^{-1/(p-1)}}{R}^{p-1}
$$
and the {\it Reverse H\"older characteristic of $w$} by
$$
[w]_{RH_p}:=\sup_R \av{w^p}{R}^{1/p} \av{w}{R}^{-1},
$$
where both suprema are taken over rectangular boxes $R$ with edges parallel to axis. If $[w]_p<\infty$, we have $w\in A_p^*$ and if $[w]_{RH_p}<\infty$, we have $w\in RH_p^*$.
\end{defin}
In \cite{perez1} it was proved that if $\mu$ is an absolutely continuous Radon measure on $\R^n$ and $[w]_p<\infty$, then for some $q>1$ we have $[w]_{RH_{q}}<C<\infty$ with an explicit dimension free estimates on $q$ and $C$. It is of a particular importance that we can take 
$$
q=1+\frac{1}{2^{p+2}[w]_p}.
$$
To prove this result, the authors used a clever version of the Calder\'on--Zygmund decomposition from \cite{Lerner1}. The aim of this paper is to derive a sharp result from the one-dimensional case for Lebesgue measure (i.e., for the classical $A_p$ and $RH_p$ classes on $\R$). In this case the result from \cite{perez1} can be obtained, for example, by means of a so-called Bellman function; i.e., a function of two variables that satisfies certain boundary and concavity conditions in its domain. In the one-dimensional case this function is known explicitly, see \cite{vasyunin1}. It has been understood for some time that, for classes of functions like $A_p$, $RH_p$ or $BMO_p$, when we work with their strong multi-dimensional analogs (e.g., $A_p^*$ and $RH_p^*$), the one-dimensional Bellman function should prove the higher-dimensional results with dimension free constants. For the Lebesgue measure and the inclusion $RH^*_p \subset A_q^*$, this was done in \cite{wall}. The trick of using the Bellman function for one-dimensional problems was also used in \cite{volberg1}, \cite{petermichl1} and \cite{treil1} (in a slightly different setting, the same trick was also used in \cite{volberg2}). 
In this paper, we present a simple version of this trick for general measures; we prove the result from \cite{perez1} as well as all other results of self-improving type for strong Muckenhoupt and Reverse H\"older weights.

\bigskip

\textbf{Acknowledgements:} we are very grateful to Vasiliy Vasyunin for helpful discussions and suggestions on the presentation of this paper.

\section{Statement of the main result}
\subsection{Properties of Muckenhoupt weights $A_p^*$}
For $p_1:=-1/(p-1)$ and every $t\in [0,1]$ define 
$
u^{\pm}_{p_1}(t)
$
to be solutions of the equation
$$
(1- u)(1-p_1 u)^{-1/p_1}=t. 
$$
The function $u^+_{p_1}$ is decreasing and maps $[0,1]$ onto $[0, 1]$; the function $u^-_{p_1}$ is increasing and maps $[0,1]$ onto $[1/p_1, 0]$. 
For a fixed $Q>1$, define
\begin{equation}\label{etomi}
s_{p_1}^{\pm}=s_{p_1}^{\pm}(Q):=u^{\pm}_{p_1}(1/Q).
\end{equation}

Our first main result is the following.
\begin{theorem}\label{thmain1}
Let $\mu$ be a Radon measure on $\R^n$  with $\mu(H)=0$ for every hyperplane $H$ orthogonal to one of the coordinate axis. Fix numbers $p>1$ and $Q>1$ and set $p_1:=-1/(p-1)$. Then for every weight $w$ with $[w]_p=Q$ we have 
$$
w\in A_q^*, \; \; \; \; \; 1-s^-_{p_1}(Q)<q<\infty,
$$ and 
$$w\in RH^*_q, \;\;\;\; 1\leqslant q < 1/s^+_{p_1}(Q),$$
where $s_{p_1}^\pm(Q)$ are defined in \eqref{etomi}.  
These ranges for $q$ are sharp for $n=1$ and $\mu=dx$.
\end{theorem}

\subsection{Properties of Reverse H\"older weights}
For $p>1$ and every $t\in [0,1]$ we define $v^{\pm}_p(t)$ to be solutions of the equation 
$$
(1- pv)^{1/p}(1- v)^{-1}=t. 
$$
In this case, $v^+_p$ is a decreasing function that maps $[0,1]$ onto $[0, 1/p]$ and $v^-_p$ is an increasing function that maps $[0,1]$ onto $[-\infty, 0]$. As before for a fixed $Q>1$, we define 
\begin{equation}\label{neprohodili}
s_{p}^{\pm}=s_{p}^{\pm}(Q):=v^{\pm}_{p}(1/Q).
\end{equation}
Our second main result concerning Reverse H\"older weights is the following.
\begin{theorem}\label{thmain2}
Let $\mu$ be a Radon measure on $\R^n$  with $\mu(H)=0$ for every hyperplane $H$ orthogonal to one of the coordinate axis. Fix numbers $p>1$ and $Q>1$. Then for every weight $w$ with $[w]_{RH_p}=Q$ we have
$$
w\in A_q^*, \;\;\;\;\; 1-s_p^-(Q)<q<\infty,
$$
and
$$
w\in RH^*_q, \;\;\;\;\; 1 \leqslant q < 1/s_p^+(Q),
$$
where $s_p^\pm(Q)$ are defined in \eqref{neprohodili}. These ranges for $q$ are sharp for $n=1$ and $\mu=dx$.
\end{theorem}
\section{Proof of the main results}
We begin with the following Theorem from \cite{vasyunin1}. This theorem ensures the existence of a certain Bellman function for a one-dimensional problem. In what follows, by letters without sub-indices (e.g., $x$, $x^\pm$) we denote points in $\R^2$ and by letters with sub-indices we denote the corresponding coordinates (e.g., $x^+_1$ denotes the first coordinate of $x^+$).  	
\begin{theorem}[Theorem 1 in \cite{vasyunin1}] Fix $p>1$ and set $p_1:=-1/(p-1)$. Also fix an $r\in (1/s_{p_1}^-, p_1]\cup [1, 1/s^+_{p_1}) $ for $s_{p_1}^\pm(Q)$ defined in \eqref{etomi}. For every $Q>1$ there exists a non-negative function $B_Q(x)$ defined in the domain $\Omega_Q:=\{x=(x_1, x_2)\in \R^2\colon 1\leqslant x_1 x_2^{-1/p_1} \leqslant Q\}$ with the following property: $B_Q(x)$ is continuous in $x$ and $Q$, and for any line segment $[x^-, x^+]\subset \Omega_Q$ and $x=\lambda x^- + (1-\lambda)x^+$, $\la\in [0,1]$, we have 
$$
B_Q(x)\geqslant \la B_Q(x^-) + (1-\la)B_Q(x^+).
$$
Moreover, $B(x_1, x_1^{p_1})=x_1^r$ and $B_Q(x)\leqslant c(r, Q) x_1^r$ for some positive constant $c(r, Q)$ and every $x\in \Omega_Q$. 
\end{theorem}
To use the concavity property of the function $B_Q$ for our proof, we need the following lemma. Its proof is given in \cite[Lemma4]{vasyunin1} with an interval instead of the rectangle; however, the proof remains the same in our case.
\begin{lemma}\label{lemmatech2}
Let the measure $\mu$ be as before. Fix two numbers $Q_1>Q>1$ and a rectangular box $R\subset \R^n$ with edges parallel to the axis. For every coordinate vector $\bf{e}$, there exists a hyperplane $H$ normal to $\bf{e}$ that splits $R$ into two rectangular boxes $R^1$ and $R^2$ with the following properties:
\begin{enumerate}
\item For $i=1,2$ we have $\mu(R^i)/\mu(R)\in (c, 1-c)$ for some constant $c\in (0,1)$;
\item For every weight $w$ with $[w]_p\leqslant Q$, we have $[x^1, x^2]\subset \Omega_{Q_1}$ and, therefore,
$$
B_{Q_1}(x_1, x_2) \geqslant \frac{\mu(R^1)}{\mu(R)} B_{Q_1}(x_1^1, x_2^1)+\frac{\mu(R^2)}{\mu(R)} B_{Q_1}(x_2^i, x_2^2),
$$
where
\begin{align*}
x_1=\av{w}{R}, \;\;\;\;\;\;\;\; & x_2 = \av{w^{p_1}}{R} \\
x_1^i = \av{w}{R^i}, \;\;\;\;\;\;\;\;  & x_2^i = \av{w^{p_1}}{R^i}.
\end{align*}
\end{enumerate}
\end{lemma}
We are ready to prove our main result.

\begin{proof}[Proof of Theorem \ref{thmain1}]
Fix a rectangular box $R$ with edges parallel to the axis, and take any $Q_1>Q$. We first explain how we split $R$ into two rectangular boxes. Take one of the $(n-1)$-dimensional faces of $R$, call it $R_{n-1}$, that has the largest $(n-1)$-area.  Among all $(n-2)$-dimensional faces of $R_{n-1}$, take one of those (call it $R_{n-2}$) that have the largest $(n-2)$-area. We proceed like this to get $R_{n-1}$, \ldots, $R_1$. Now take a vector ${\bf e}$ that is orthogonal to every $R_i$, $i=1,\ldots, n-1$.\footnote{In the case $n=2$, we just take ${\bf e}$ orthogonal to the longest side of $R$.} We now split $R$ according to Lemma \ref{lemmatech2}. Notice that all the corresponding $i$-dimensional faces of $R^1$ and $R^2$ have smaller $i$-areas than the corresponding $i$-dimensional faces of $R$.
We now take the boxes $R^1$ and  $R^{2}$ and repeat the same procedure. If we repeat this $M$ times, we get a family of rectangular boxes $\mathcal{R}=\{R^{i, M}\}_{i=1\ldots 2^{M}}$. Denote
\begin{align*}x_1&=\av{w}{R}, \;\;\;\;\;\;\;\; &x_2 &= \av{w^{p_1}}{R} \\
x_1^{i,M} &= \av{w}{R^{i,M}}, \;\;\;\;\;\;\;\;  &x_2^{i,M} &= \av{w^{p_1}}{R^{i,M}}.
\end{align*}
Abusing the notation, we also define step-functions
$$
x_1^{M}(t):=\sli_{i=1}^{2^{M}} x_1^{i,M} \mathbbm{1}_{R^{i,M}}(t), \;\;\;\;\;\;\;\; x_2^{M}(t):=\sli_{i=1}^{2^{M}} x_2^{i,M} \mathbbm{1}_{R^{i,M}}(t).
$$
From the construction of rectangular boxes, we notice that $x_1^M(t) \to w(t)$ and $x_2^M(t)\to w^{p_1}(t)$ as $M\to \infty$ for $\mu$-a.e. $t\in R$. Indeed, our splitting procedure (and the fact that we have $\mu(R^i)/\mu(R)\in (1-c, c)$ at every step) guarantees that 
$$
\max_{i=1,\ldots, 2^M} \textup{diam}(R^{i,M}) \to 0, \; \; M\to \infty,
$$
and we obtain the convergence of $x_1^M(t)$ and $x_2^M(t)$ from the Lebegue differentiation theorem for Radon measures. Therefore,
$$
B_{Q_1}(x_1, x_2)\geqslant \sli_{i=1}^{2^{M}} \frac{\mu(R^{i,M})}{\mu(R)} B_{Q_1}(x_1^{i,M}, x_2^{i,M}) = \frac{1}{\mu(R)}\ili_R B_{Q_1}(x_1^{M}(t), x_2^M(t))d\mu(t).
$$
By the Fatou lemma,
\begin{multline}
B_{Q_1}(x_1, x_2)\geqslant \frac{1}{\mu(R)}\ili_R \lim_{M\to \infty} B_{Q_1}(x_1^M(t), x_2^M(t))d\mu(t) =\\ \frac{1}{\mu(R)}\ili_R B_{Q_1}(w(t), w^{p_1}(t))d\mu(t) = \frac{1}{\mu(R)} \ili_R w^r(t)dt = \av{w^r}{R}.
\end{multline}
Since $B_{Q_1}(x_1, x_2)$ is continuous in $Q_1$ and the above estimate holds for any $Q_1>Q$, we get
$$
c(r, Q)\av{w}{R}^r=c(r, Q)x_1^r \leqslant \av{w^r}{R}.
$$
If we use this estimate for $q=r \in [1, 1/s^+_{p_1})$, we obtain $w\in RH^*_q$. If we use this estimate for $-1/(q-1)=r\in (1/s_{p_1}^-, p_1]$, we obtain $w\in A^*_q$ for $1-s^-_{p_1}(Q)<q<\infty$.
\end{proof}

To prove Theorem \ref{thmain2} we need to use a different Bellman function $B_Q$. Namely, the following result holds.
\begin{theorem}[Theorem 1 in \cite{vasyunin1}] Fix $p>1$ and $r\in (1/s_{p}^-, 1]\cup [p, 1/s^+_{p}) $ for $s^\pm_p$ defined in \eqref{neprohodili}. For every $Q>1$ there exists a non-negative function $B_Q(x)$ defined in the domain $\Omega_Q:=\{x=(x_1, x_2)\in \R^2\colon 1\leqslant x_1 x_2^{-1/p} \leqslant Q\}$ with the following property: $B_Q(x)$ is continuous in $x$ and $Q$, and for any line segment $[x^-, x^+]\subset \Omega_Q$ and $x=\lambda x^- + (1-\lambda)x^+$, $\la\in [0,1]$, we have 
$$
B_Q(x)\geqslant \la B_Q(x^-) + (1-\la)B_Q(x^+).
$$
Moreover, $B(x_1, x_1^{p})=x_1^r$ and $B_Q(x)\leqslant c(r, Q) x_1^r$ for some positive constant $c(r, Q)$ and every $x\in \Omega_Q$. 
\end{theorem}
We also notice that the analog of Lemma \ref{lemmatech2} reads the same, and with this in hand, the proof of Theorem \ref{thmain1} is analogous to the proof of Theorem \ref{thmain2}; we leave the details to the reader.
\bibliography{ref}
\bibliographystyle{plain}
\end{document}